\documentclass[11pt]{amsart}

\usepackage{graphicx, amsmath, amsfonts, amsthm, amssymb, fullpage, caption, subcaption, microtype, verbatim, color}
\newtheorem{thm}{Theorem}[section]
\newtheorem{cor}[thm]{Corollary}
\newtheorem{prop}[thm]{Proposition}
\newtheorem{lem}[thm]{Lemma}

\newcommand{\Z}{\mathbb{Z}}
\newcommand{\N}{\mathbb{N}}
\newcommand{\pic}[1]{\begin{minipage}{0.83in} \[\includegraphics[width=0.8in]{{#1}} \]\end{minipage}}

\title{Roger and Yang's Kauffman Bracket Arc Algebra is Finitely Generated}

\author{Martin Bobb}
\address{Department of Mathematics, University of Texas at Austin, Austin, TX 78712, USA}

\author{Dylan Peifer}
\address{Department of Mathematics, Cornell University, Ithaca, NY 14853, USA}

\author{Helen Wong}
\address{Department of Mathematics, Carleton College, Northfield, MN 55057, USA}

\thanks{The authors were supported in part by NSF Grant DMS-1105692. }

\begin{document}

\maketitle

\begin{abstract}
We provide an explicit, finite set of generators for the Kauffman bracket arc algebra defined by Roger and Yang.  
\end{abstract}

%%%%%%%%%%%%%%%%%%%%%%%%%%%%%%%%%%%%%%%%%%%%%%%%%%
%%%--- Introduction ---%%%
%%%%%%%%%%%%%%%%%%%%%%%%%%%%%%%%%%%%%%%%%%%%%%%%%%
%\section{Introduction}

%Low-dimensional topology was drastically changed in the 1970s  and 1980s by two seemingly separate events.  
%First was William Thurston's Geometrization Conjecture, which stated that every 3-manifold can be cut into pieces which are locally isometric to one of eight homogeneous geometries and which was eventually proved by Perelman.  The most prevalent of these geometries is hyperbolic geometry, and the study of hyperbolic structures on 3-manifolds remains an active area of research.   
%The second development was V. Jones's  introduction of the Jones polynomial for knots and links and E. Witten's reinterpretation using mathematical physics.   This new area of quantum topology linked low-dimensional topology with many other areas of mathematics, including representation theory, quantum groups, and mathematical physics.  Based on the physical intuition, it is widely believed that the Jones polynomial and related invariants encode hyperbolic geometric information.  However,  

It is widely believed that quantum topology and  hyperbolic geometry are deeply related \cite{KashaevVolume, MurakamiMurakami, KashaevInv, KBB,  Hikami, GarouSlope, LeSlope}, but  there are few existing connections between them which are  concretely phrased and well-understood.    
Recently defend by J. Roger and T. Yang in \cite{RoYa14}, the \emph{Kauffman bracket skein algebra of arcs and links} (or more shortly the \emph{Kauffman bracket arc algebra}) of a punctured surface might be one point of attack.   

The Kauffman bracket arc algebra of a punctured surface generalizes the \emph{Kauffman bracket skein algebra}, which was defined by Turaev \cite{Tu88} and Przytycki \cite{Pr91} based on L. Kauffman's skein theoretic description of the Jones polynomial in \cite{Ka87}. Like the skein algebra, the arc algebra is defined combinatorially.  It is generated by  unions of framed links and framed arcs whose endpoints are at the punctures of the surface.  The arc algebra involves four relations---two are the usual skein relations from the skein algebra and another two are versions of those involving arcs and punctures.

%Doug Bullock, Charles Frohman, Joanna Kania-Bartoszyn?ska, Understanding the Kauffman bracket skein module, J. Knot Theory Ramifications 8 (1999), 265?277.
%[BFK2] Doug Bullock, Charles Frohman, Joanna Kania-Bartoszyn?ska, The Kauffman bracket skein as an algebra of observables, Proc. Amer. Math. Soc. 130 (2002), 2479?2485.

Roger and Yang designed the  arc algebra so that it is  a quantization of the decorated Teichm\"uller space studied by \cite{Penner} and others.    The combined work in  \cite{TuraevPoisson, BullockFrohmanJKB, BullockFrohmanJKB2, PrzSikora} established an explicit relationship between the skein algebra and  the Teichm\"uller space from hyperbolic geometry by interpreting the skein algebra of a hyperbolic surface as a quantization of the $\mathrm{PSL}_2(\mathbb C)$-character variety of the fundamental group for that surface.   
The Roger and Yang approach follows the same line for the arc algebra, replacing the Atiyah-Bott-Goldman Poisson structure of the character variety by the Weil-Petersson Poisson structure on the decorated Teichmuller space identified by Mondello \cite{Mondello}.  

The purpose of this paper is to show that the arc algebra is finitely generated, and to provide an explicit set of generators.  This  generalizes Bullock's result in \cite{Bu99} for the skein algebra.    Interestingly, when a surface has five or more punctures, the generating set for the arc algebra presented here is fewer in number than Bullock's generating set for the skein algebra.  In \cite{BobbPeiferKennedyWong}, the authors also show that the arc algebra is finitely presented for some small surfaces. 
\smallskip

{\bf Acknowledgements.}  The authors would like to thank Stephen Kennedy, Francis Bonahon and Tian Yang for helpful discussions, and the Carleton College mathematics department for their support throughout this research.    The authors would also like to thank the anonymous reviewer for pointing out a mistake in an earlier version.  
%This work was completed as a part of the senior comprehensive exercise of the first two authors.  

%%%%%%%%%%%%%%%%%
%%%--- The Arc Algebra ---%%%
%%%%%%%%%%%%%%%%% 
\section{The arc Algebra}

Let $F_{g,n}$ denote a compact, orientable surface of genus $g$ with $n$ points (the punctures) removed.    Let $A$ be an indeterminate, with formal square roots $A^{\frac12}$ and $A^{-\frac12}$.   In addition, let there be an indeterminate $v_i$ associated to the $i$th  puncture.  Let $R_n = \Z[A^{\pm\frac{1}{2}}][v_1^{\pm1}, v_2^{\pm1}, \dots, v_n^{\pm1}]$  denote the ring of Laurent polynomials in the commuting variables $A^{\frac12}$ and $v_1, \ldots, v_n$.      

A framed curve in the thickened surface $F_{g,n} \times [0,1]$ is the union of  framed knots and framed arcs that go from puncture to puncture.   See \cite{RoYa14} for a precise definition.     The \emph{arc algebra} $\mathcal A(F_{g,n})$ is generated as an $R_n$-module by framed curves in $F_{g,n} \times [0,1]$, up to isotopy and subject to the following four relations:
\begin{align*}
&1)
\quad
\begin{minipage}{.5in}\includegraphics[width=\textwidth]{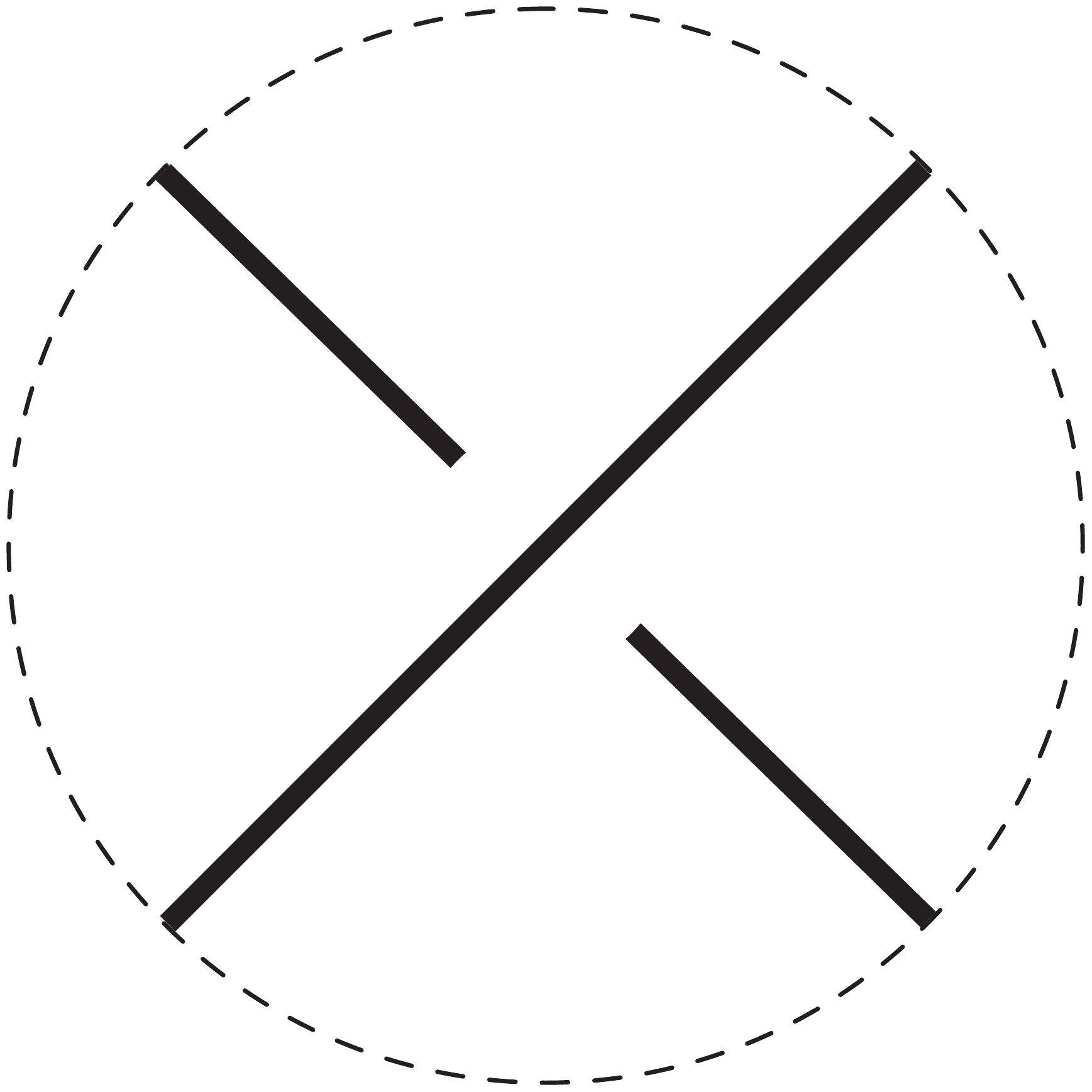}\end{minipage} 
=  A\begin{minipage}{.5in}\includegraphics[width=\textwidth]{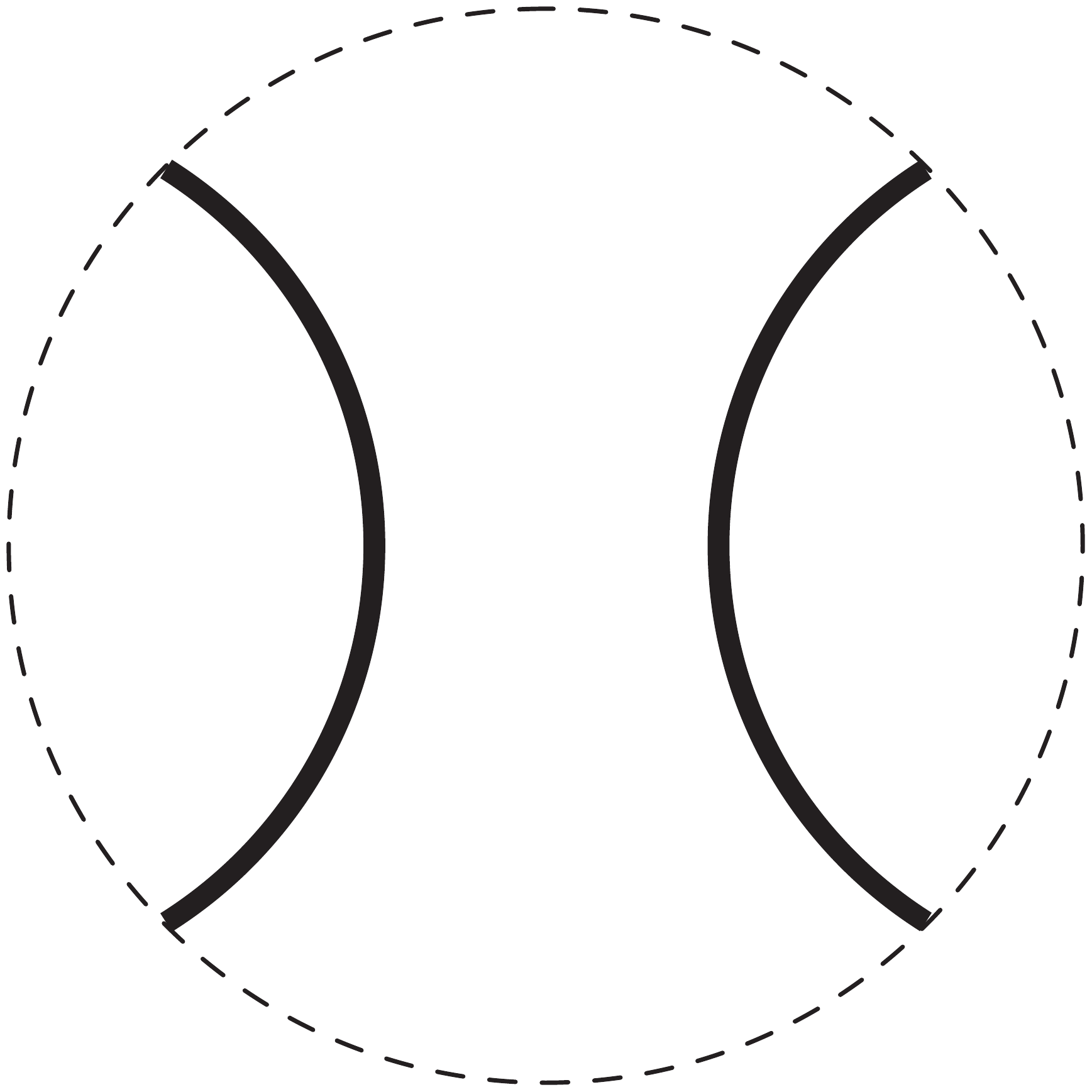}\end{minipage} 
+A^{-1}\begin{minipage}{.5in}\includegraphics[width=\textwidth]{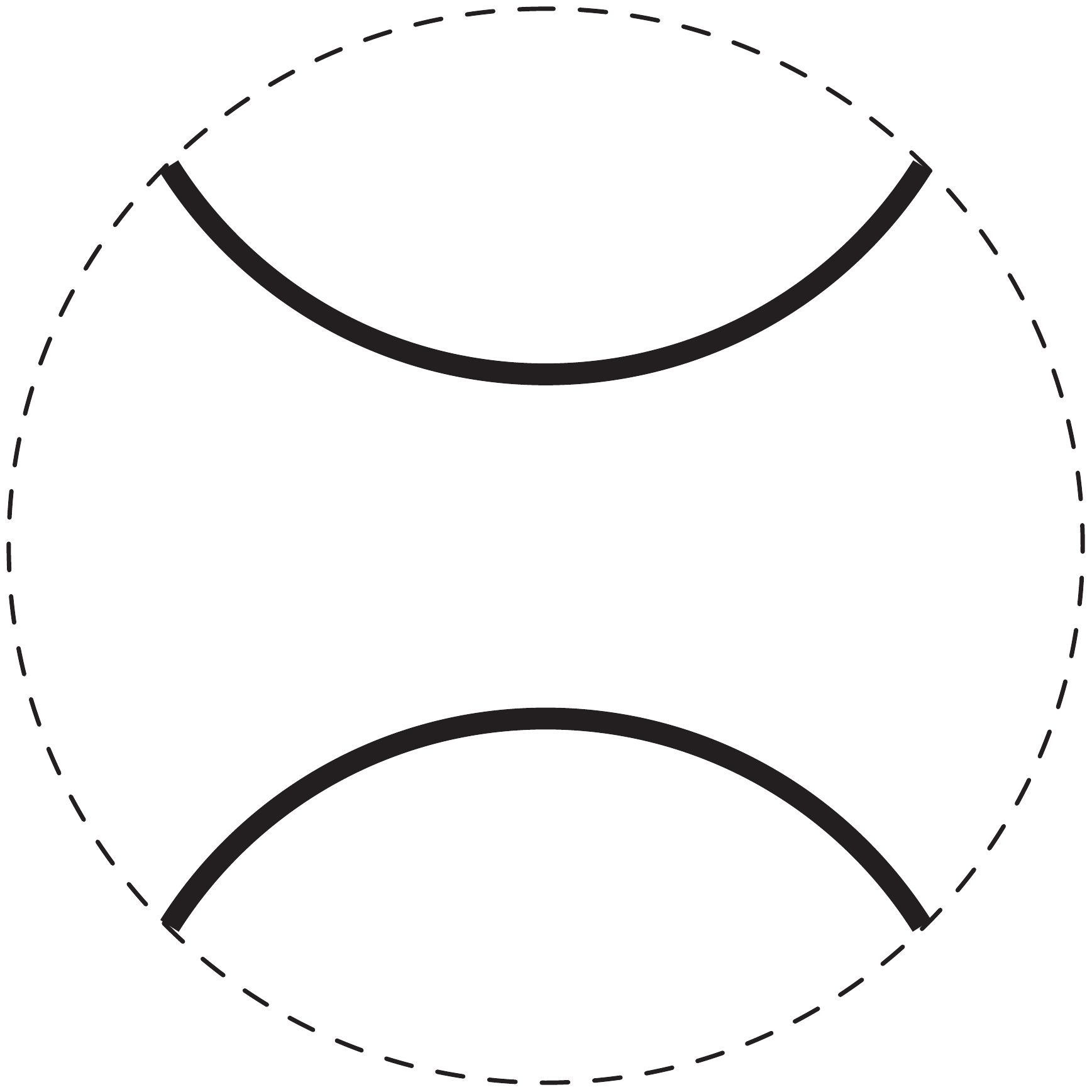}\end{minipage}  
&\text{Skein Relation}\\
&2)
\quad 
v_i \begin{minipage}{.5in}\includegraphics[width=\textwidth]{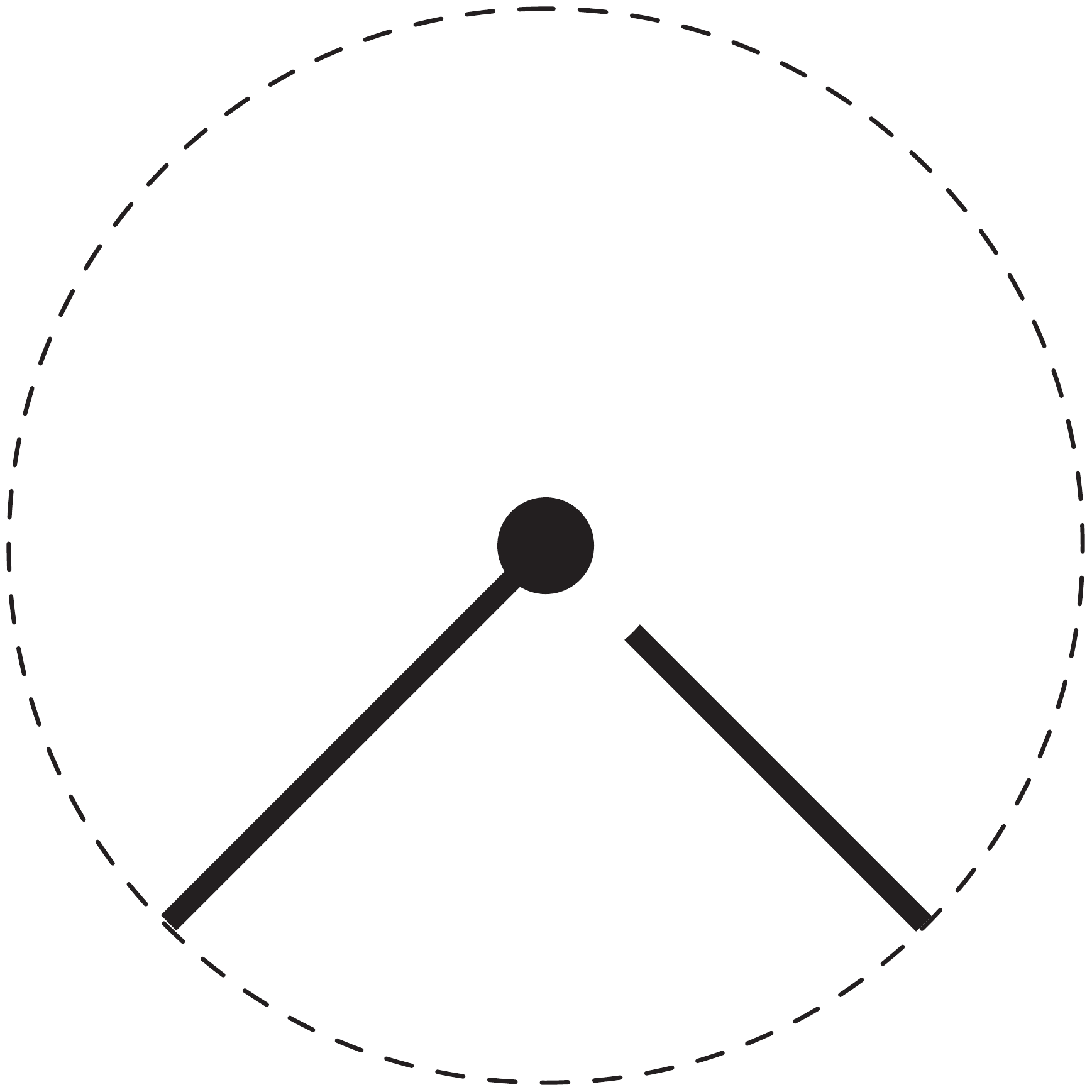}\end{minipage} 
=  A^\frac{1}{2}\begin{minipage}{.5in}\includegraphics[width=\textwidth]{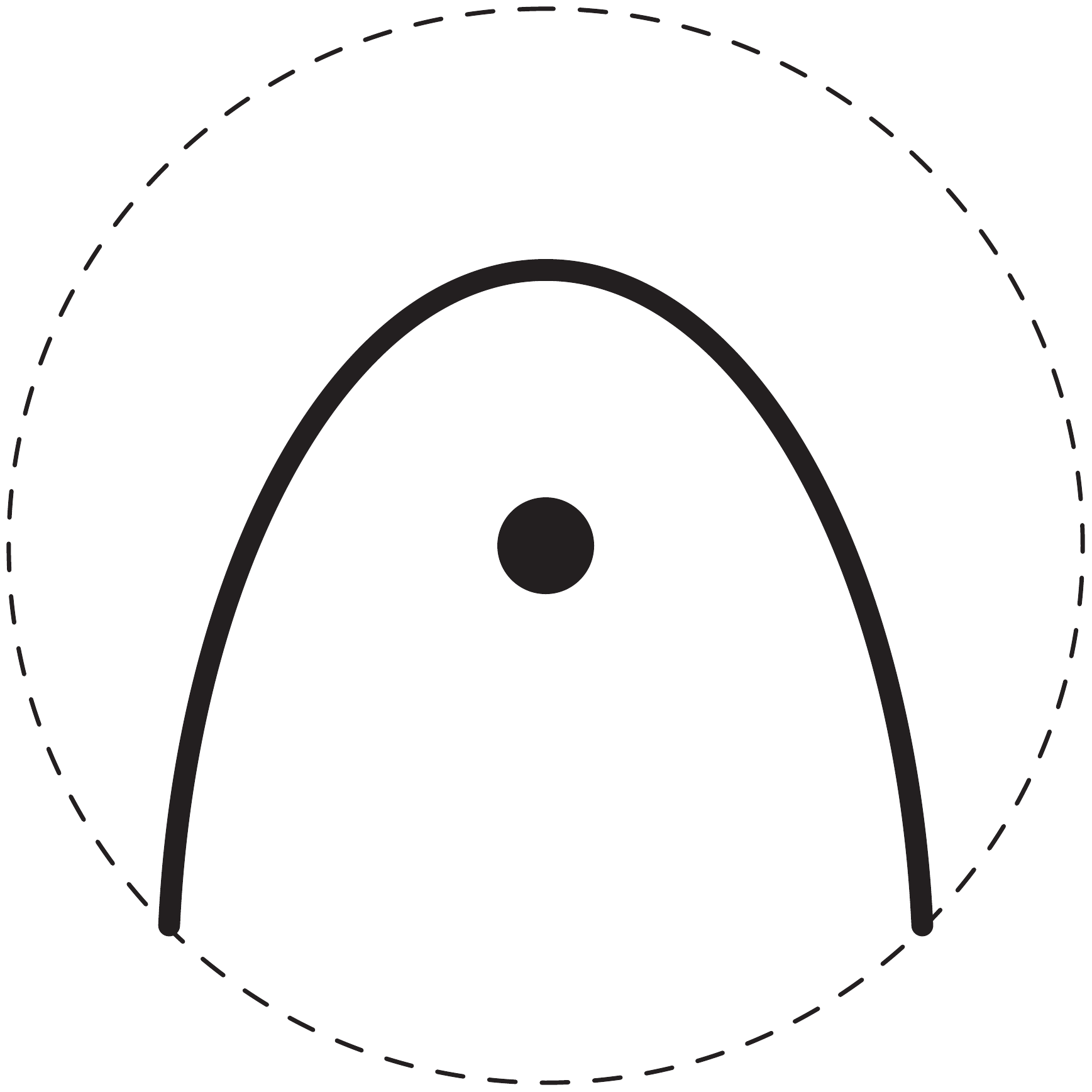}\end{minipage} + A^{-\frac{1}{2}}\begin{minipage}{.5in}\includegraphics[width=\textwidth]{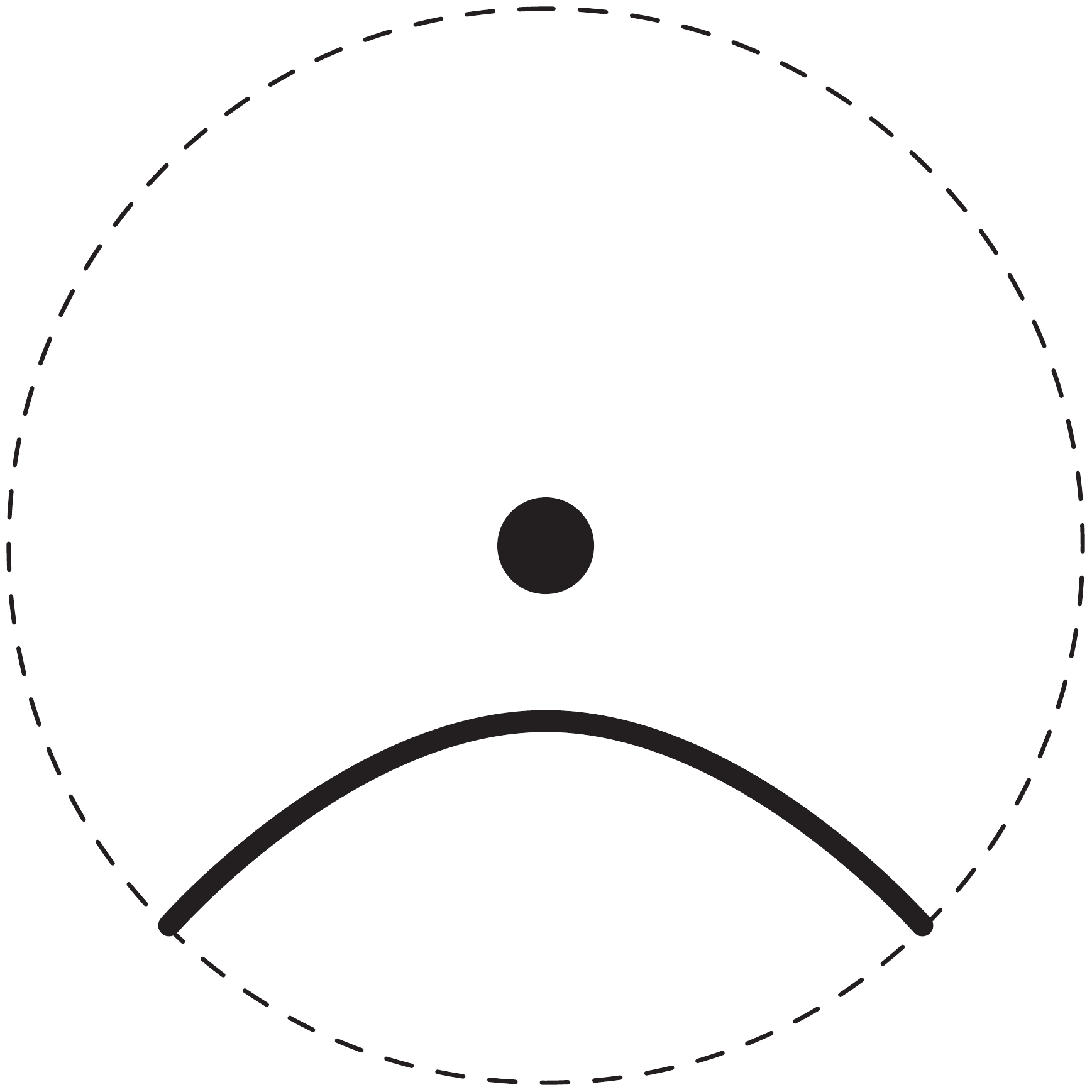}\end{minipage}  
&\text{Puncture-Skein Relation on $i$th puncture} \\
&3)
\quad 
\begin{minipage}{.5in}\includegraphics[width=\textwidth]{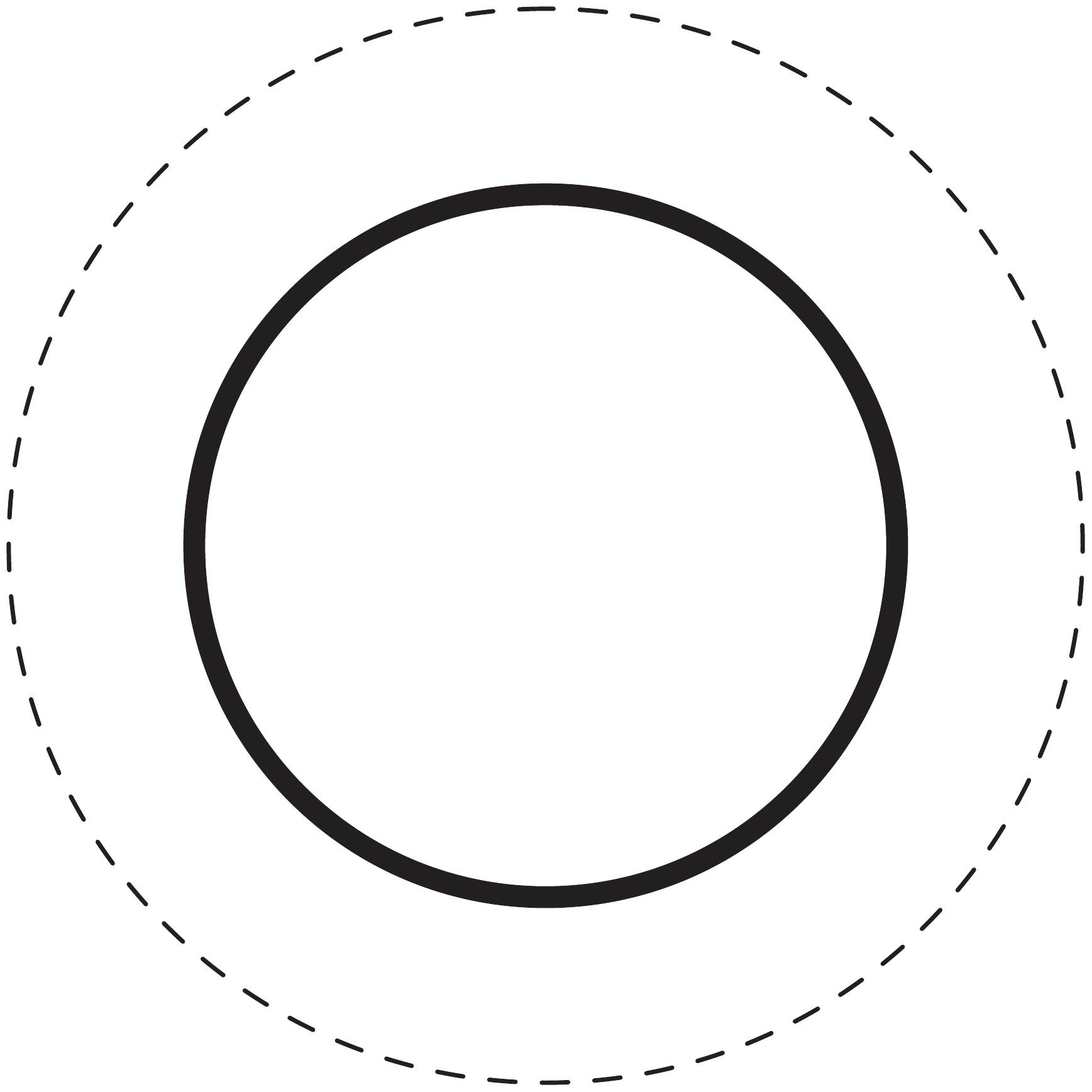} \end{minipage} 
= - A^2 - A^{-2} 
&\text{Framing Relation} \\
&4)
\quad 
\begin{minipage}{.5in}\includegraphics[width=\textwidth]{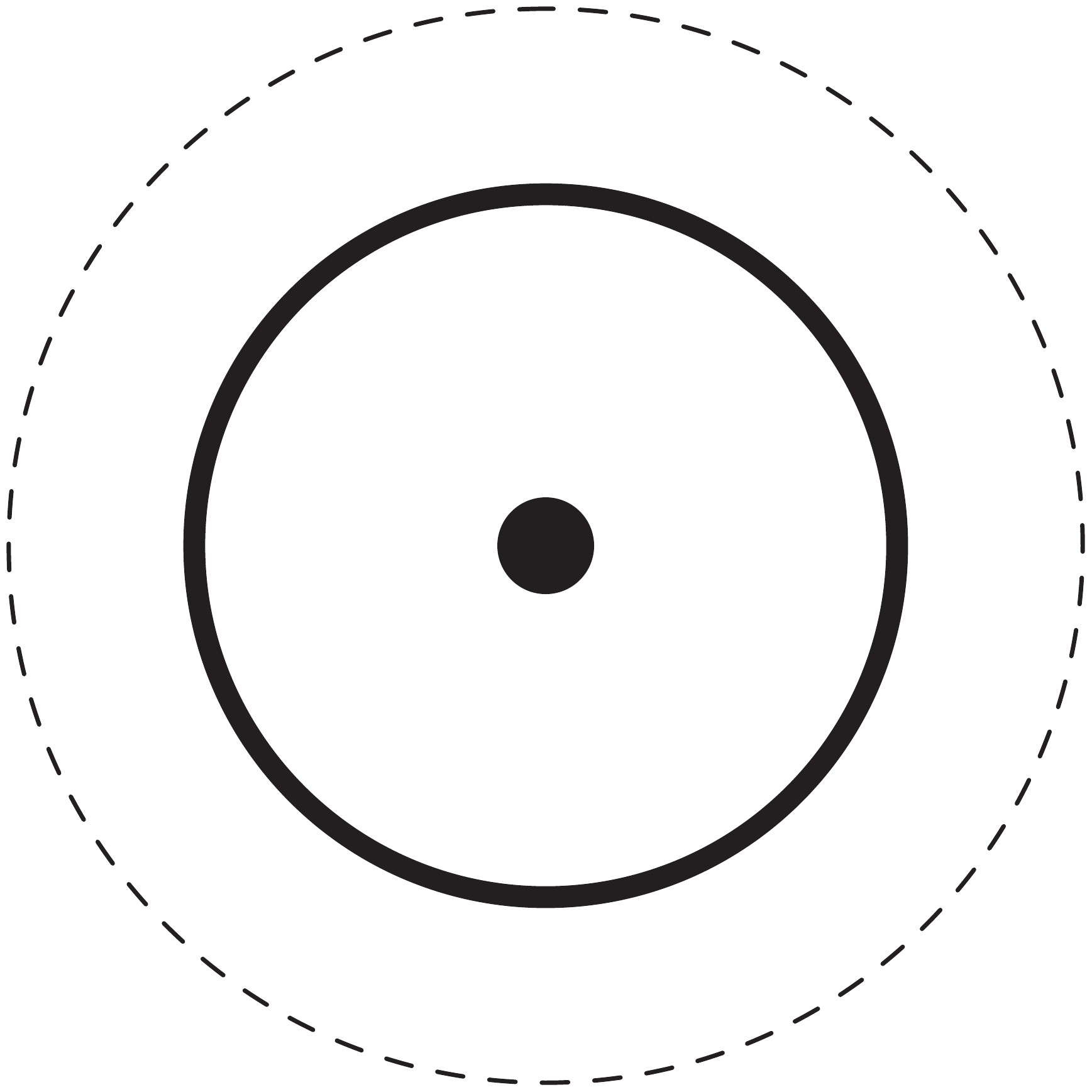} \end{minipage} 
= A + A^{-1} 
&\text{Puncture-Framing  Relation} 
\end{align*}
The links in any one relation are assumed to be identical outside of the small balls depicted.  The framings are vertical.  
The algebraic structure of $\mathcal A(F_{g,n})$ is induced by stacking elements.  That is, if $[L_1], [L_2] \in \mathcal A(F_{g,n})$ are respectively represented by framed curves $L_1, L_2$ in $F_{g,n} \times[0, 1]$, the product $[L_1]\ast[L_2] = [L_1' \cup L_2'] \in \mathcal A(F_{g,n})$ is represented by the union of the framed link $L_1' \subset F_{g,n} \times [0, \frac12]$ obtained by rescaling $L_1 \subset F_{g,n} \times [0, 1]$ and of  the framed link $L_2' \subset F_{g,n} \times [\frac12, 1]$ obtained by rescaling $L_2 \subset F_{g,n} \times [0, 1]$.

The arc algebra is very closely related to the skein algebra defined by Turaev \cite{Tu88} and Przytycki \cite{Pr91}.  Recall that the skein algebra $\mathcal S(F_{g,n})$ is the $\Z[A, A^{-1}]$-algebra generated by framed links in 
$F_{g,n} \times[0,1]$, up to isotopy and relations 1 and 3 above, and its multiplication is also by stacking elements.  
By exploiting the relationship between the skein algebra and the arc algebra and then applying the result of Bullock in \cite{Bu99}, the authors obtained the following proposition in \cite{BobbPeiferKennedyWong}. 
\begin{prop} \label{prop:n01}
When $n = 0$ or $1$, the arc algebra $\mathcal A(F_{g, n})$ is generated by $4^{g}-1$ knots.  
\end{prop}
In this paper, we turn to the cases when $n >1$, for which the arguments in \cite{Bu99} alone no longer suffice.   The remainder of this paper will be devoted to proving the following theorem.  

\begin{thm} \label{thm:finitegen}
For $n > 1$ the arc algebra $\mathcal A(F_{g,n})$ is generated by a set of $n(4^{g}-1)$ knots and $\frac{n(n-1)}{2} ( 4^{g})$ arcs.
\end{thm}

%%%%%%%%%%%%%%%%%%%%%%%%%%%%%%%%%%%%%%%%%%%%%%%%%%
%%%--- Finite Generation ---%%%
%%%%%%%%%%%%%%%%%%%%%%%%%%%%%%%%%%%%%%%%%%%%%%%%%%
\section{Finite Generators}

We define a \emph{simple knot} to be a framed closed curve that allows a projection without any crossings and that does not bound a disk containing one or no punctures.   A \emph{simple arc} is a framed, connected arc that allows a projection without any crossings and whose endpoints are at two distinct punctures.  A \emph{simple curve} is either a simple knot or a simple arc.   
The simple curves form a basis for $\mathcal A(F_{g,n})$, regarded as an $R_n$-module (see \cite{BobbPeiferKennedyWong} or \cite{RoYa14}).    So, a set of elements generating the simple curves necessarily generates the entire algebra $\mathcal A(F_{g,n})$.  

We prove Theorem \ref{thm:finitegen} inductively using complexity functions, which are defined in Section~\ref{subsec:Def}.    We show that the simple curves are generated by ones with small complexity, and this is proved in a series of lemmas in Section~\ref{subsec:Lemmas}.     An explicit description of the finitely  many small complexity curves that generate  $\mathcal A(F_{g,n})$ is provided in Section~\ref{sec:proof}.    

Many of the ideas in our proof were inspired by Bullock's paper \cite{Bu99}, with a few nearly identical.  We will only briefly repeat arguments as necessary for completeness, and highlight modifications and new ingredients whenever appropriate.  

%%%%%%%%%%%%%%
%%%--- Definitions ---%%%
%%%%%%%%%%%%%%\
\subsection{Three measures of complexity} \label{subsec:Def}
Let $F_{g,n}^*$ be $F_{g,n}$ with a small disk $D$ removed.  Note that any framed curve in $F_{g,n}$ can be isotoped to avoid $D$.  Thus any set that generates $\mathcal A(F_{g,n}^*)$ will also generate $\mathcal A(F_{g,n})$.  

Choose a generalized handle decomposition of $F_{g,n}^*$ of the type indicated in Figure~\ref{fig:Fgn*}.  In particular, the decomposition consists of one handle of index $0$ and  $2g+n$ handles of index $1$; the co-cores of the 1-handles are indicated by dashed lines in Figure~\ref{fig:Fgn*}.   The $1$-handles are grouped into $g$ separate pairs of overlapping handles that are disjoint from the punctures, and $n$ handles that each go around one puncture.  The puncture handles do not overlap with each other nor with any of the pairs of overlapping handles.  Let $G$ be the set of $2g$ co-cores corresponding to the overlapping handles, and let $P$ be the set of $n$ co-cores corresponding to the puncture handles.  \begin{figure}[htpb]
\centering
\includegraphics[width=5.55in]{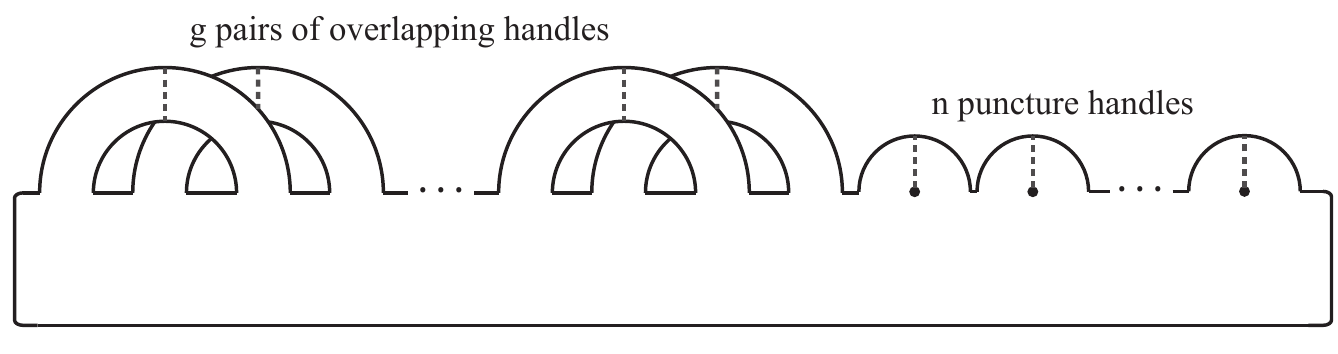}
\caption{Generalized handle decomposition of $F_{g,n}^*$. The dashed lines indicate the co-cores. }
\label{fig:Fgn*}
\end{figure}

Unless otherwise indicated, we assume that curves are framed  and are isotoped so that they intersect the co-cores $G \cup P$ minimally.   
We now define three ways to measure the complexity of a curve $\gamma$.  
To begin with, let \emph{handle complexity} $h(\gamma)$ be the number of times $\gamma$ intersects the co-cores minus the number of co-cores it intersects, i.e., $h(\gamma)= |\gamma \cap (G \cup P)| -  | \{ c \in G \cup P \; | \; \gamma \cap c \neq 0 \} |$.  For instance, if $h(\gamma) = 0$, then $\gamma$ intersects each co-core either once or not at all.   

%As we will see in Section \ref{subsec:Lemmas}, such zero complexity curves generate $\mathcal A(F_{g,n}^*)$, but many of them are redundant.  

%In particular, we may restrict to zero handle complexity curves $\gamma$ which intersect both co-cores of a pair of handles in a prescribed way.  

The second measure of complexity ensures that a curve intersects pairs of overlapping handles in a prescribed way.  As in \cite{Bu99}, a pair of overlapping handles of $F_{g,n}^*$ is called a \emph{good pair} for $\gamma$ if every time $\gamma$ passes through both handles in that pair, it does so as in Figure~\ref{fig:good}; otherwise, the pair of handles is a \emph{bad pair} for $\gamma$.  Note that if $\gamma$  passes through only one of an overlapping pair of handles (possibly many times), that pair of handles remains a good pair for $\gamma$;  bad pairs can occur only if the curve intersects both handles in the pair.   Let  the \emph{bad pair complexity} $b(\gamma)$ be the number of bad pairs of handles for $\gamma$.   

%We will show that the generating set will consist  of curves with $h(\gamma) = 0 $ and $b(\gamma) = 0$.  

%Define the \emph{bad pair complexity} $ b(\gamma)$ to be the number of bad pairs of handles for $\gamma$.  
%
\begin{figure}[htpb]
\centering
\includegraphics[width=1.5in]{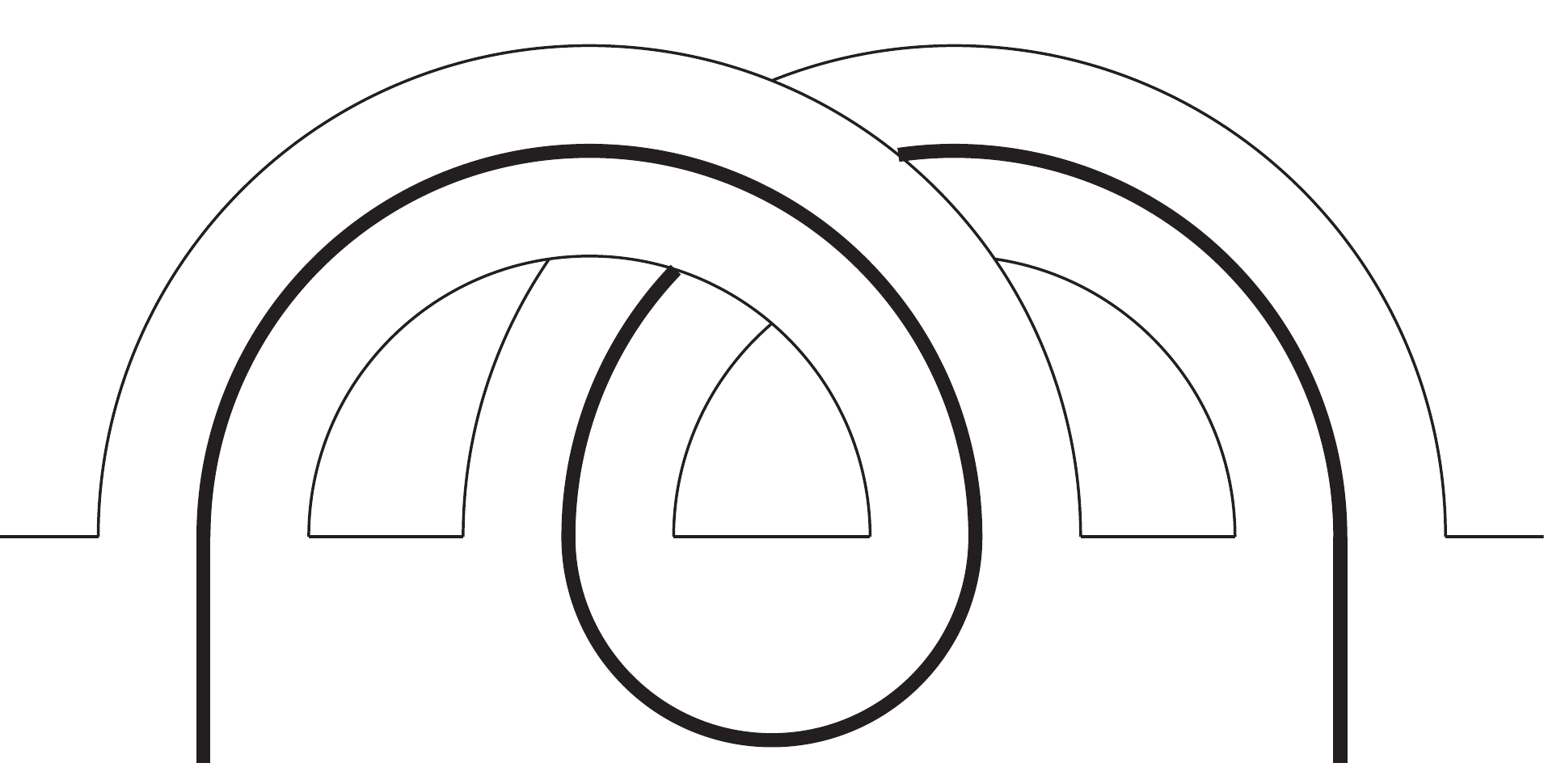}
\caption{A good pair of overlapping handles.}
\label{fig:good}
\end{figure}

Finally, let the \emph{puncture complexity} $r(\gamma)$ be the total number of times $\gamma$ intersects the co-cores of the puncture handles, i.e. $r(\gamma) = | \gamma \cap P|$.

In the following sections, we will show that a generating set for the arc algebra consists of simple curves with $h = 0 $, $b= 0$, and $r = 0 $ or $1$.  
Our proof is by induction on the \emph{total complexity} $|\gamma|$, which is the triple
\[ | \gamma | = (h(\gamma), b(\gamma), r(\gamma)) \in \N^3  \]
where $\N^3$ is ordered lexicographically. 

%\newpage

%%%%%%%%%%%%%%%%%%%
%%%--- A Series of Lemmas ---%%%
%%%%%%%%%%%%%%%%%%%
\subsection{Reducing complexity} \label{subsec:Lemmas}

The first goal is to show that the arc algebra $\mathcal A(F_{g,n}^*)$ is generated by simple curves that pass through each handle at most once.   

%This is done by generalizing an argument from Lemma 2 in \cite{Bu99} to include arcs.  Our proof will focus on the cases involving arcs and only briefly repeat arguments from \cite{Bu99} for completeness.

%The following series of lemmas show that $\gamma$  is in the set generated by simple arcs with total complexity $(0,0,0)$ and simple knots $K$ with total complexity $(0,0,0)$ or $(0,0,1)$. 

%These lemmas will be used as the steps in the theorem.
%Individually, they will show that we may reduce any of the three parameters of complexity for a curve by writing curves as combinations of curves with smaller complexity values.
%Taken together, the lemmas will show by strong induction that any curve in the arc algebra may be written in terms of a specific generating set.

%The following lemma shows that if we take a simple curve $\gamma$ in the arc algebra, we may reduce the parameter $h(\gamma)$ by writing the curve as a combination of simple curves with smaller handle complexity.
%This proof is taken from \cite{Bu99} with the addition of a single subcase to handle arcs.

%%%--- lem:handlec ---%%%
\begin{lem} \label{lem:handlec}
$\mathcal A(F_{g,n}^*)$ is generated by simple curves with handle  complexity $h=0$.
\end{lem}
\begin{proof}
Let $\gamma$ be a simple curve with handle complexity $h(\gamma) > 0$.    Then $\gamma$ intersects the co-core of some handle  $H$ at least twice.   

First consider when $H$ is one of an overlapping pair of handles, or when $H$ is a puncture handle and $\gamma$ does not have an endpoint at that puncture.  Then the  arguments from Lemma~2 and Lemma~3 from \cite{Bu99} directly apply, since the curves involved in the calculation differ only in a neighborhood of two adjacent strands of $\gamma$ in the handle $H$.  In particular, the punctures are outside this neighborhood and play no role, meaning that the proof for the skein algebra is essentially the same as the proof for the arc algebra.  It follows  that $\gamma$ can be written as a linear combination of products of simple knots and arcs with strictly lower handle complexity.  

Thus assume $H$ is a puncture handle and $\gamma$ ends at the puncture at $H$.  Consider the neighborhood of the two strands of $\gamma$ closest to the puncture at $H$.  There are a few cases: the four depicted in Figure~\ref{fig:Lem21cases} and their mirror images.  %
\begin{figure}[htpb]
 \begin{subfigure}[b]{0.9in}
    \pic{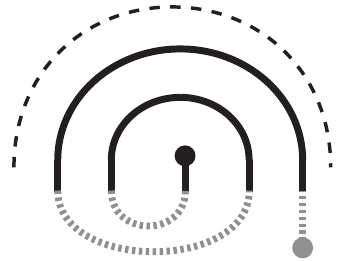}
                \caption{}
                \label{fig:Lem21Case1}
        \end{subfigure}
 \begin{subfigure}[b]{0.9in}
    \pic{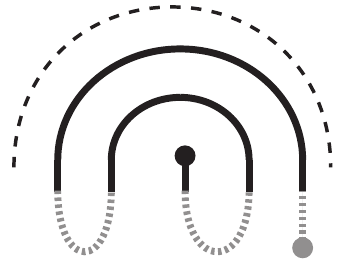}
                \caption{}
                \label{fig:Lem21Case3}
        \end{subfigure}
 \begin{subfigure}[b]{0.9in}
    \pic{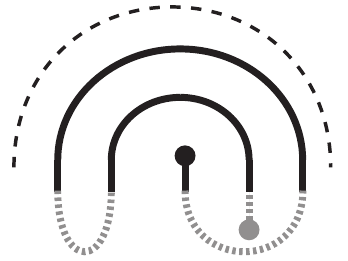}
                \caption{}
                \label{fig:Lem21Case2}
        \end{subfigure}
 \begin{subfigure}[b]{0.9in}
    \pic{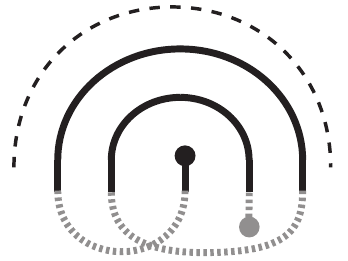}
                \caption{}
                \label{fig:Lem21Case4}
        \end{subfigure}
\caption{Configurations where $\gamma$ is a simple arc ending at the puncture of $H$. }
\label{fig:Lem21cases}
\end{figure}
The two strands of $\gamma$ closest to the puncture are in black, and the dotted grey lines indicate how these two strands are connected outside of this neighborhood.    
In particular, although dotted grey lines might look like they intersect  in Case (D), they in fact do not; $\gamma$ itself does not have any crossings.  
We will only do the cases in Figure~\ref{fig:Lem21cases} and leave their mirror images as an exercise for the reader.  

In Case (A), \
\[ \pic{Lem21Case1.pdf}   = - A^{2} \pic{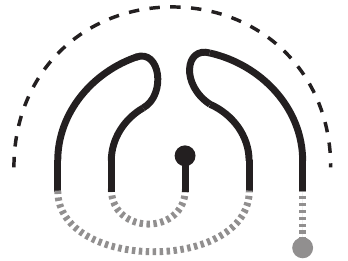} + A \pic{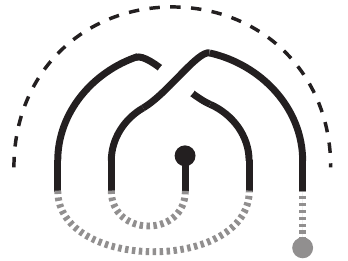}   .\]
The curve in the first term on the right is a simple arc that intersects the co-core of $H$ twice fewer than $\gamma$.    The curve in the second term is a product of a simple knot and a simple arc, both of which intersect the co-core of $H$ once fewer than $\gamma$.     Thus $\gamma$ is written as a linear combination of products of simple curves with strictly lower handle complexity.

We start off Case (B) similarly, 
\[ \pic{Lem21Case3.pdf}   =  - A^{2} \pic{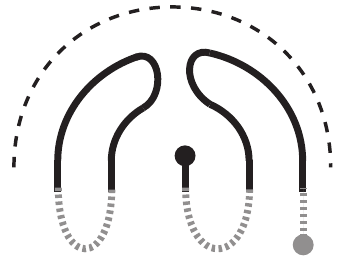} + A \pic{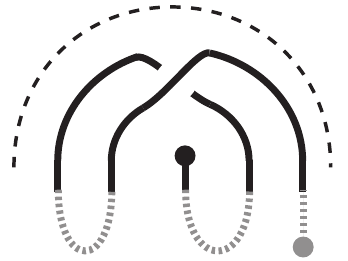} . \]
The curve in the first term is a product of a simple knot and a simple arc, both with handle complexity less than $h(\gamma)$.   The curve in the second term is not itself a simple arc, but
\[ \pic{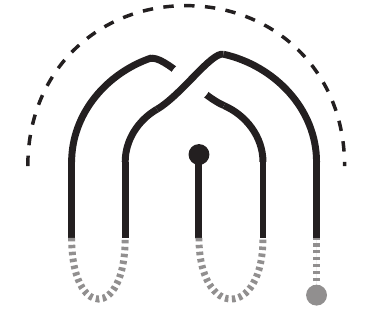} = -A^2\pic{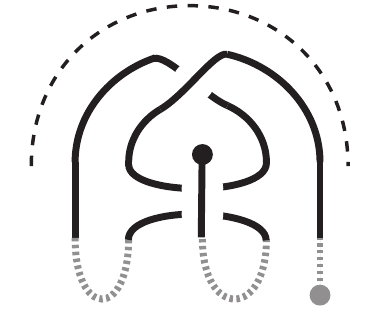} +A\pic{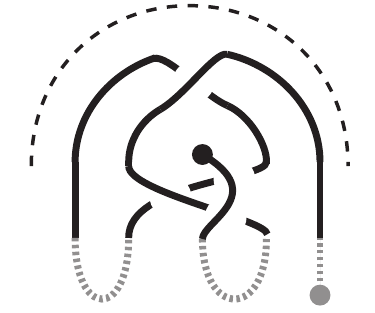}.
\]
Moreover,
\begin{align*}\pic{pzeta3p.pdf} 
&= A \pic{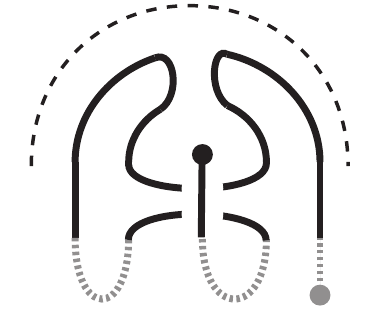} + A^{-1} \pic{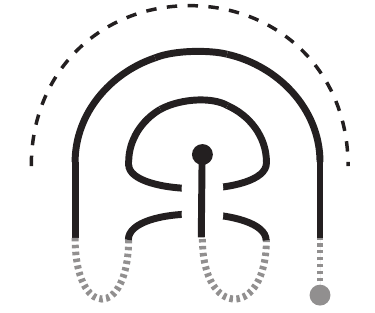}   
\\
&= A \pic{pzeta3pB.pdf} + A^{-1} \pic{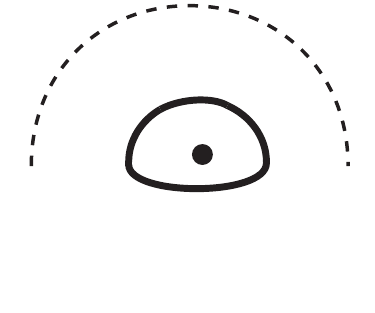}\!\ast\!\pic{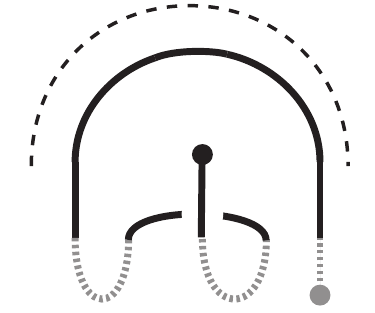},    
\end{align*}
and further applications of the  skein relation at the remaining crossings will show it to be generated by simple curves with lower handle complexity than $\gamma$. 
The other term involves a product,  
\[
\pic{pzeta4p.pdf} = \pic{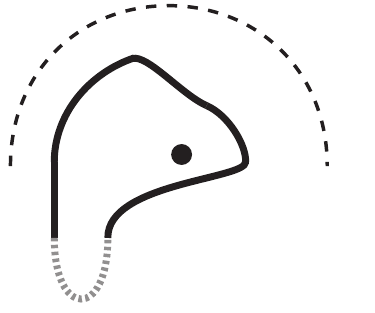}\!\ast\!\pic{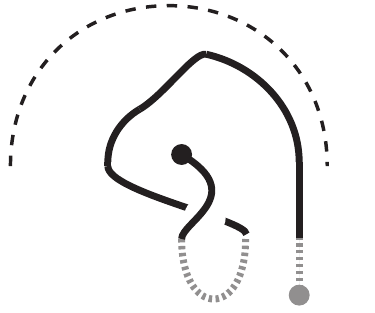}, 
\]
and each of those two components can be generated by simple curves with lower complexity.   By back-substituting, we can write $\gamma$ using curves with handle complexity less than $h(\gamma)$.  

The calculation for the Case (C) is nearly identical to Case (B):
\[\pic{Lem21Case2.pdf}   
=  - A^{2} \pic{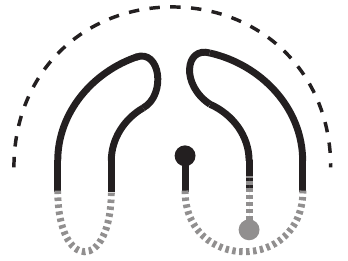} + A \pic{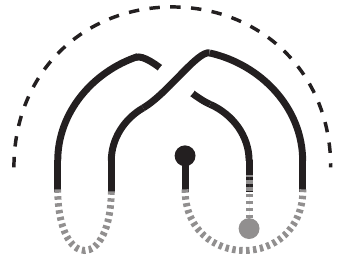}\]
and
\[ \pic{Lem21Case2o.pdf}
=
-A^2 \left( A \pic{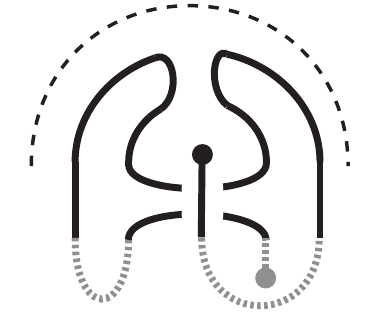} + A^{-1} \pic{pzeta3pAbot.pdf}\!\ast\!\pic{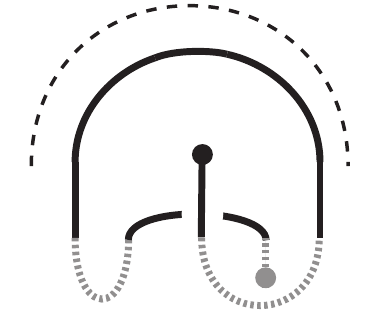} \right) + A \pic{pzeta4pbot.pdf}\!\ast\!\pic{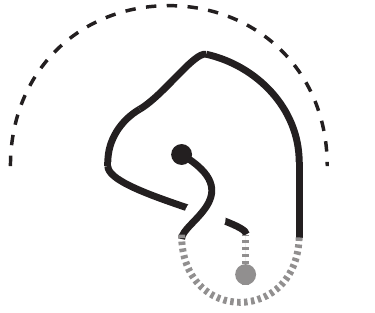}.
\]
Each of the diagrams with crossings can be further simplified by a straightforward application of the skein relation.  The diagrams involved will have lower complexity than $h(\gamma)$.  

Finally, Case (D) is nearly identical to Case (A): \
\[ \pic{Lem21Case4.pdf}   = - A^{2} \pic{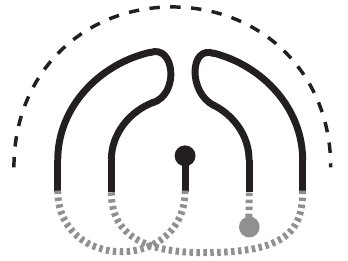} + A \pic{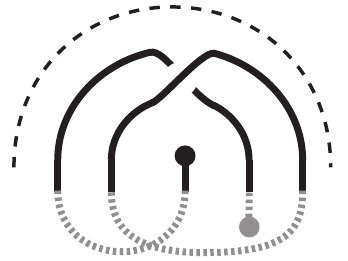}   .\]
The curve in the first term on the right is a simple arc, and the curve in the second term is a product of a simple knot and a simple arc.  Each of those simple curves has strictly lower handle complexity.    

By inductively repeating the above arguments, we see that $\gamma$ can be generated by curves of zero handle complexity.  
 \end{proof}

Although a curve with zero handle complexity passes through each handle at most once, it may still go from one handle to its overlapping pair in a circuitous way.  However, such curves can be generated by curves that traverse both handles in an overlapping pair in the most direct, ``good'' way depicted in Figure \ref{fig:good}.    We cite this result as the next lemma but will omit its proof.  The straightforward calculation from Lemma 1 in  \cite{Bu99} involves diagrams which agree outside a neighborhood of a pair of handles, and thus applies for the arc algebra with very little modification.  

 %%%--- lem:pairc ---%%%
\begin{lem} \label{lem:pairc}
$\mathcal A(F_{g,n}^*)$  is generated by simple curves with both zero handle complexity and zero bad pair complexity.  
\end{lem}

We thus focus on simple curves with both handle complexity $h=0$ and bad pair complexity $b = 0$, i.e. with total complexity $(0,0, r)$.   The next three lemmas will show that such curves are themselves generated by  ones with puncture complexity $r= 0$ or 1.    This is done by ``pulling curves across punctures'' using the puncture-skein relation of the arc algebra, a move that has no counterpart in the skein algebra.  The argument is different for arcs and knots:  
Lemma~\ref{lem:puncturecknots} addresses knots, Lemma~\ref{lem:puncturecarcs} addresses arcs, and
Lemma~\ref{lem:puncturec} then combines the results.

%%%--- lem:puncturecknots ---%%%
\begin{lem} \label{lem:puncturecknots}
A simple knot with total complexity $(0, 0, r)$ and $r >1$ can be expressed in terms of simple arcs and simple knots with strictly lower total complexity.
\end{lem}

\begin{proof}
Let $\gamma$ be a simple knot with $r(\gamma) > 1$.  Because $h(\gamma)= 0$, $\gamma$ must pass through at least two distinct puncture handles, say those corresponding to puncture $i$ and puncture $j$.  Like we did before, we keep track of how the strands of $\gamma$ are connected outside of these two puncture handles with the use of dotted grey lines.     

In the first case, 
\begin{align*}
\gamma &= \begin{minipage}{1.5in}\includegraphics[width=0.90\textwidth]{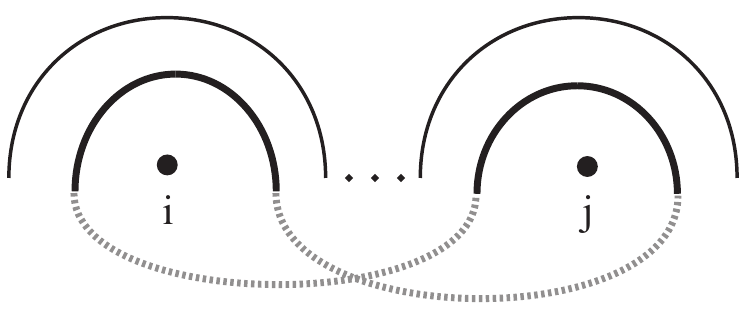}\end{minipage} \\
&= v_{i} v_{j} A^{-1} \begin{minipage}{1.5in}\includegraphics[width=\textwidth]{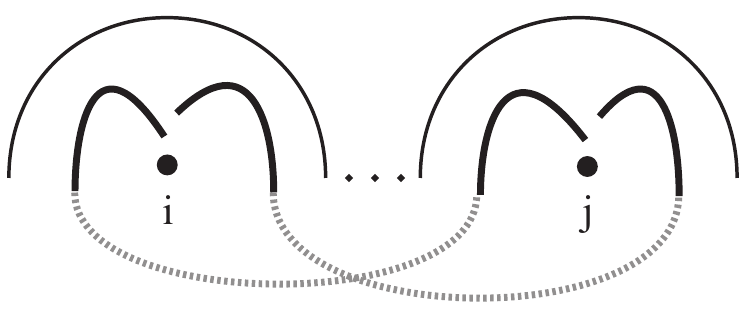}\end{minipage} 
- A^{-1} \; \begin{minipage}{1.5in}\includegraphics[width=0.90\textwidth]{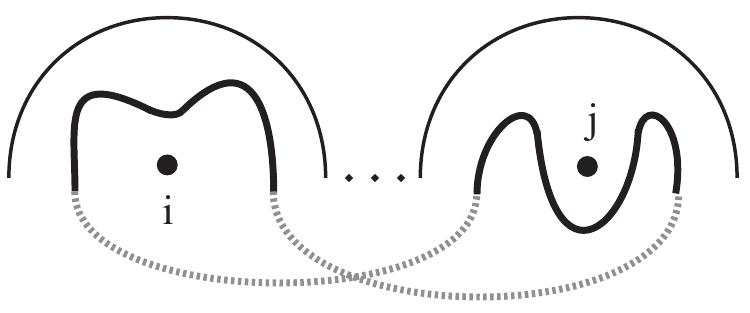}\end{minipage}  \\
& \qquad \; - A^{-1} \; \begin{minipage}{1.5in}\includegraphics[width=0.90\textwidth]{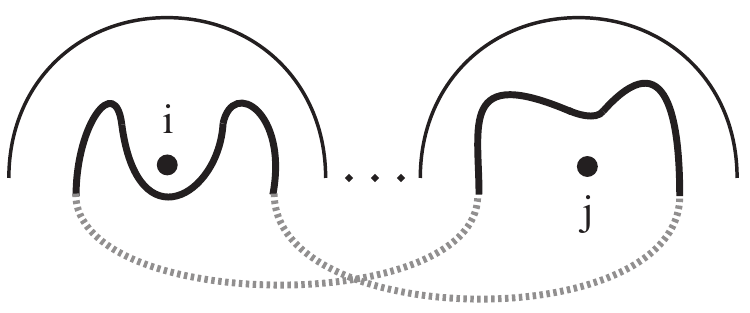}\end{minipage}  
-A^{-2} \begin{minipage}{1.5in}\includegraphics[width=0.90\textwidth]{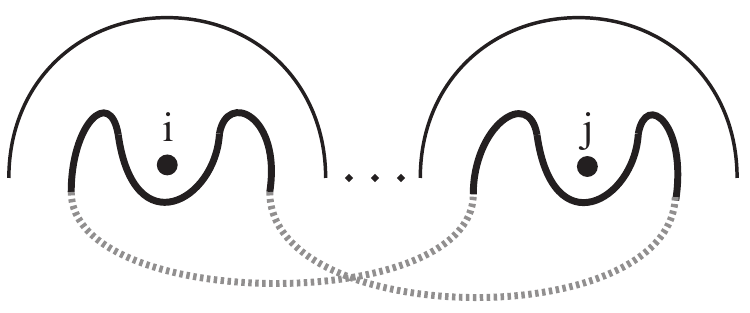}\end{minipage} .
\end{align*}
In the second case, 
\begin{align*}
\gamma &= \begin{minipage}{1.5in}\includegraphics[width=0.90\textwidth]{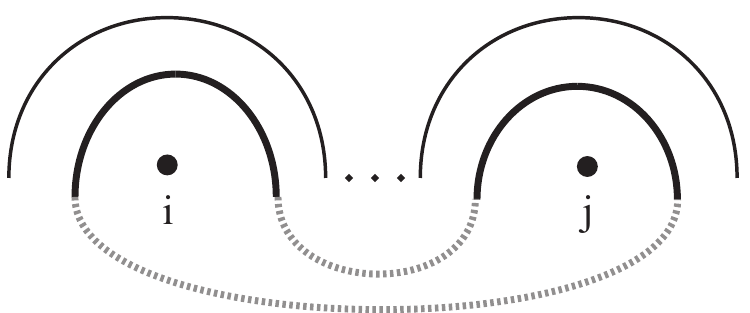}\end{minipage} \\
&= v_{i} v_{j}\begin{minipage}{1.5in}\includegraphics[width=\textwidth]{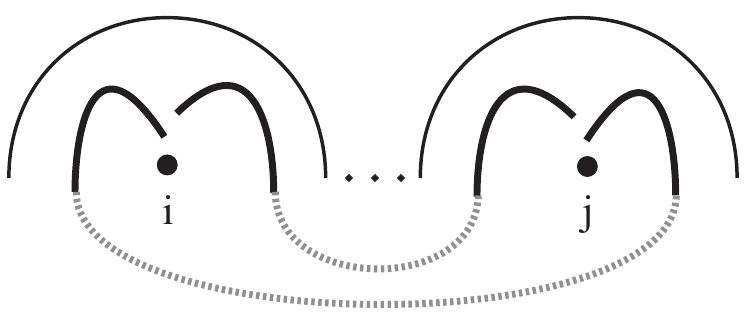}\end{minipage} 
- A \; \begin{minipage}{1.5in}\includegraphics[width=0.90\textwidth]{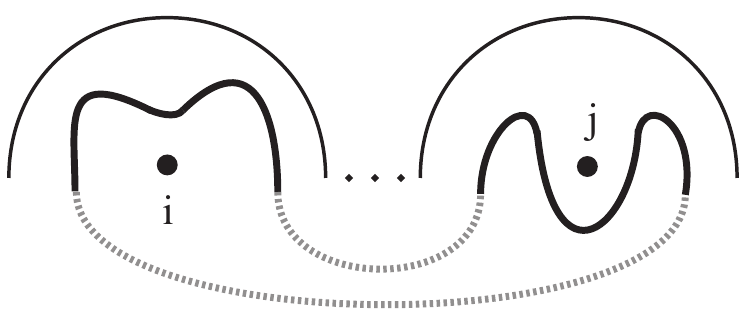}\end{minipage}\\
& \qquad \; \; - \begin{minipage}{1.5in}\includegraphics[width=0.90\textwidth]{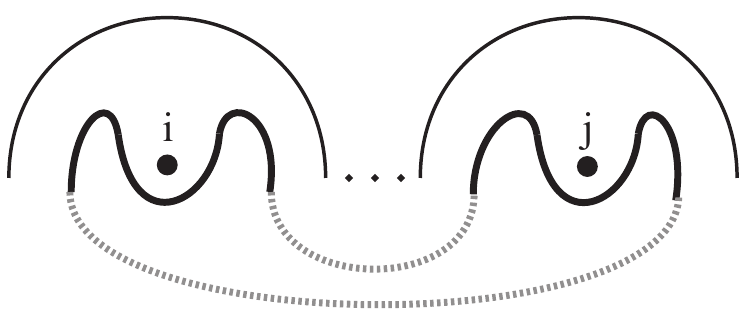}\end{minipage}  
-A^{-1} \begin{minipage}{1.5in}\includegraphics[width=0.90\textwidth]{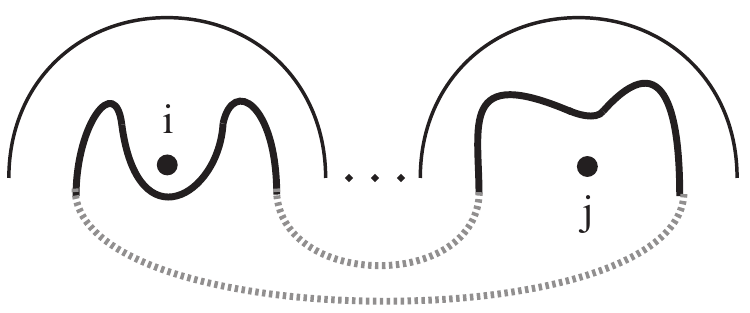}\end{minipage} .
\end{align*}
In both cases we have written $\gamma$ using simple curves and a product of two simple arcs, each with smaller puncture complexity $r$.   
\end{proof}

%%%--- lem:puncturecarcs ---%%%
\begin{lem} \label{lem:puncturecarcs}
A simple arc with total complexity $(0, 0, r)$ and $r>0$ can be expressed in terms of simple arcs with strictly lower total complexity and simple knots with equal or lower total complexity.  
\end{lem}

\begin{proof}
Suppose $r(\gamma) > 0$, so that $\gamma$ must intersect the co-core of at least one puncture handle, say that of puncture $i$.
Again, we split into a few cases.

In the first case, puncture $i$ is not an endpoint of $\gamma$.
Apply the puncture-skein relation, so
\[
\gamma = \pic{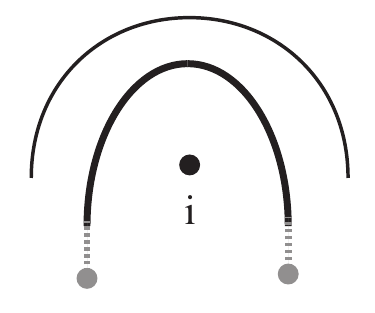} = v_i A^{-1/2}\; \pic{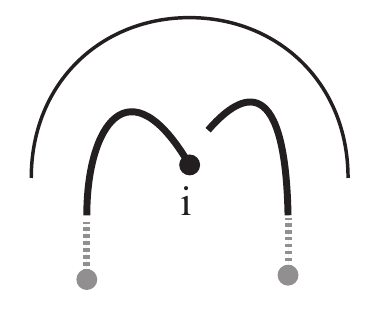} - A^{-1} \; \pic{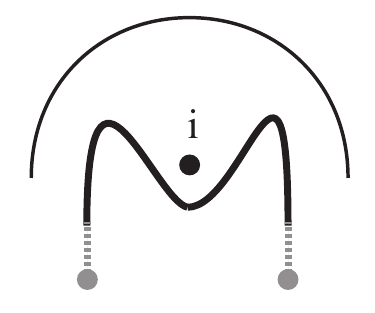}.
\]
Note that the first term on the right-hand side is a product of simple arcs, as the requirement that puncture $i$ is not an endpoint of $\gamma$ ensures that we have split the arc into two pieces, each of which intersect at least one fewer puncture handle co-core than $\gamma$ does.    The last term on the right is another arc that intersects at least one fewer puncture handle than $\gamma$.  

In the remaining cases, puncture $i$ is an endpoint of $\gamma$.
We proceed similarly, but with more caution.
Either the strand on the right or the strand on the left ends at puncture $i$. 

If the strand on the right ends at puncture $i$, we apply the puncture-skein relation to derive
\[
\pic{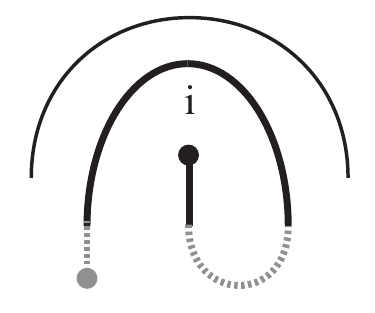} =v_iA^{-1/2}\; \pic{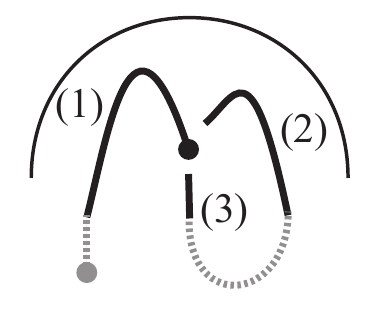} - A^{-1} \; \pic{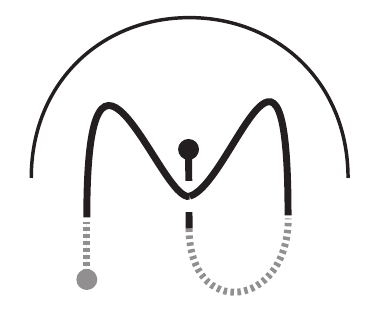}.  
\]
The numbering of the arcs on the first term on the right side of the equation indicates the relative heights at the $i$th puncture, so that $(2)$ is above $(3)$.   
It is easy to see that the second term on the right can be expressed using simple knots and arcs with strictly lower puncture complexity, by applying the skein relation on the single crossing.  The first term on the right is a product of arcs, and we apply the puncture-skein relation.  Specifically,  
\begin{align*}
\pic{puncarciarcs.pdf} &= \pic{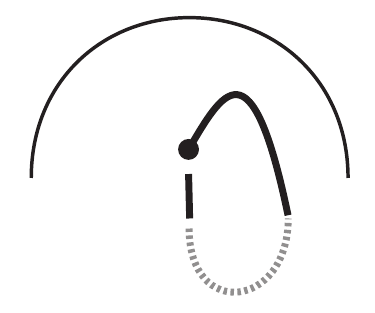} \ast \pic{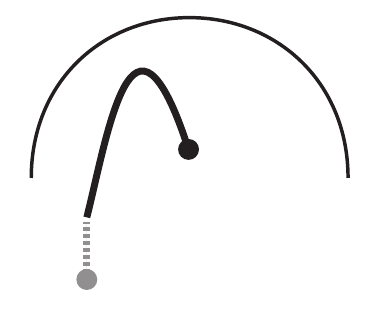}
\\
& =\frac{1}{v_i}\left(A^{\frac{1}{2}}\pic{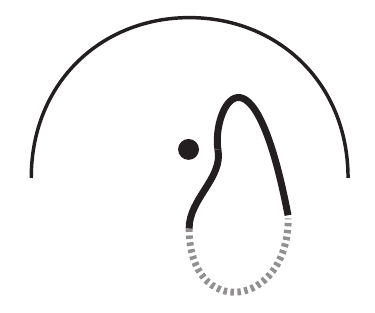} + A^{-\frac{1}{2}}\pic{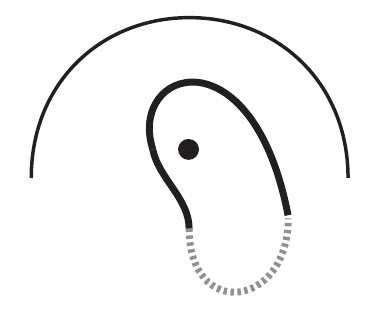}\right) \ast  \pic{puncarciarcstop.pdf}.
\end{align*}
All of the terms involve simple arcs with strictly lower complexity or simple knots with equal or lower complexity.  

If the strand of $\gamma$ on the left side of the puncture handle ends at puncture $i$, then we begin with
\[
\pic{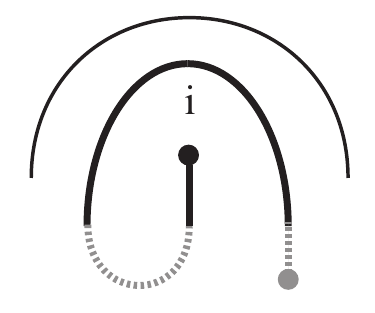} =v_iA^{1/2}\; \pic{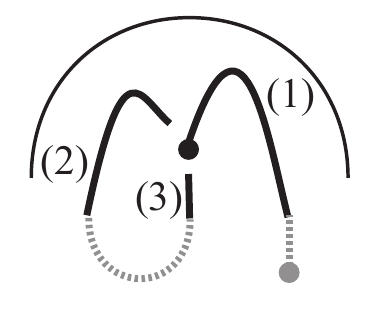} - A \; \pic{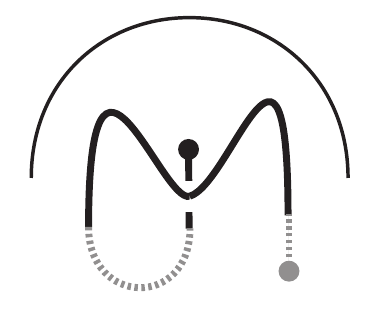}  
\]
and proceed similarly as in the previous case.

\end{proof}

%%%--- lem:puncturec ---%%%
\begin{prop} \label{lem:puncturec}
$\mathcal A(F_{g,n}^*)$  is generated by simple arcs with total complexity (0,0,0) and by simple knots with total complexity (0,0,0) or (0,0,1).  
\end{prop}

Proposition \ref{lem:puncturec} follows from putting together Lemmas \ref{lem:handlec} and \ref{lem:pairc} and a straightforward inductive argument using Lemmas ~\ref{lem:puncturecarcs} and \ref{lem:puncturecknots}.

We next show that the curves in Proposition \ref{lem:puncturec} are characterized by the handles they intersect and their endpoints.    The argument is very similar to Bullock's argument in \cite{Bu99}, so we provide only an abbreviated proof here.

\begin{lem} \label{lem:determinehandle}
A simple curve from Proposition \ref{lem:puncturec} is completely determined by the handles it intersects and, if the curve is an arc, its endpoints.  
\end{lem}

\begin{proof}
Suppose $\gamma$ has $h(\gamma) = b(\gamma) = 0$ and, if it is an arc, $r(\gamma) = 0$.  There exists a disk $\tilde D$ in $F_{g,n}^*$ (essentially its 0-handle minus small neighborhoods near the overlapping pairs of handles) such that whenever $\gamma$ exits $\tilde D$, it either passes through a puncture handle once, passes through a single handle of an overlapping pair once, passes through both handles of an overlapping pair in the good way described in Figure~\ref{fig:good}, or, if it is an arc, goes to a puncture.  Moreover, if $\gamma$ is an arc, $r(\gamma) = 0$ implies that the points of $\gamma \cap \partial \tilde D$ closest to the two endpoints of $\gamma$ are adjacent on $\partial \tilde D$.    Thus, inside $\tilde D$, there is only one way to connect up the points of $\gamma \cap \partial \tilde D$ so that $\gamma$ is a connected curve without any crossings.  
\end{proof}

\subsection{A generating set for $\mathcal A(F_{g,n})$}  \label{sec:proof}
So far we have only considered curves on $F_{g,n}^*$.  Since $F_{g,n}^*$ was obtained from $F_{g,n}$ by removing only a small disk $D$, the set of generating curves  for $\mathcal A(F_{g,n}^*)$ described in Proposition \ref{lem:puncturec} will also generate $\mathcal A(F_{g,n})$.   However, there are redundancies in that set, because there are curves which are isotopic in $F_{g,n}$ but not in $F_{g,n}^*$.

In particular, let $F_{g,n-1}^*$ be the subsurface of $F_{g,n}$ obtained by further removing the $n$th puncture and the $n$th puncture handle from $F_{g,n}^*$.  We choose the $n$th one for convenience; the following argument applies for any of the puncture handles.   Now consider a curve  $\gamma$ that passes through the $n$th puncture handle of $F_{g,n}^*$.   We can take the strands of $\gamma$ that pass through the $n$th puncture handle, push it out the part of $\partial F_{g,n}^*$ near the $n$th puncture handle, across the disk $D$, and then back through the part of $\partial F_{g,n}^*$ away from the $n$th puncture handle.  An example of such a move is depicted in Figure~\ref{fig:avoid}.  We see that $\gamma$ now avoids the $n$th puncture handle and is a curve in $F_{g, n-1}^*$.  

\begin{figure}[h]
\centering
\includegraphics[width=15cm]{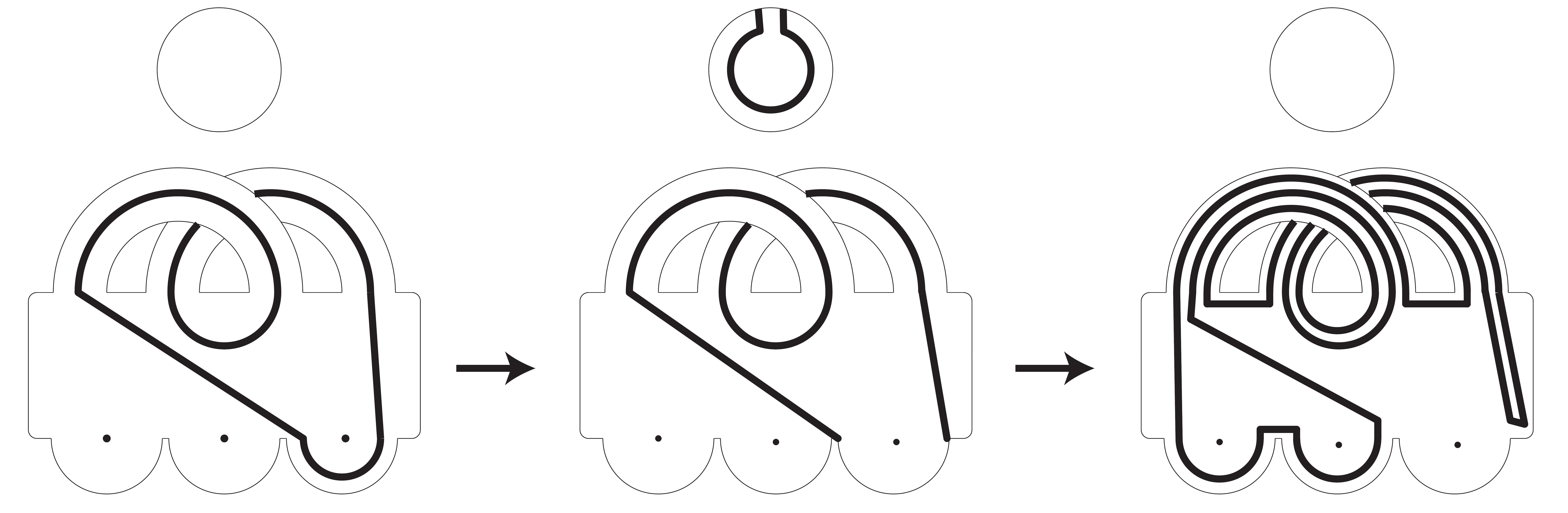}
\caption{Pushing a curve across the boundary of $F_{1,3}^*$. The disk $D$ is shown above the corresponding $F_{1,3}^*$.}
\label{fig:avoid}
\end{figure}

Let us apply the above principle to the generators from Proposition \ref{lem:puncturec} which are simple knots in $F_{g, n}^*$ with total complexity $(0,0,1)$.  We may now regard them as simple knots in $F_{g, n-1}^*$, possibly with higher total complexity in $F_{g, n-1}^*$.    The results from Section \ref{subsec:Lemmas} can be used to express such knots in terms of simple arcs with total complexity (0,0,0) and simple knots with total complexity (0,0,0) and (0,0,1), where complexity is now measured in $F_{g, n-1}^*$.    Since simple arcs and simple knots with total complexity (0,0,0) in $F_{g,n-1}^*$ also have total complexity (0,0,0) in $F_{g,n}^*$, we have thus proven the following theorem.

\begin{thm} \label{thm:simplegenerators}
The arc algebra $\mathcal A(F_{g,n})$ is generated by simple arcs and simple knots which lie in $F_{g,n}^*$ with total complexity (0,0,0); and by simple knots which lie in $F_{g,n-1}^*$ with total complexity (0,0,1). 
\end{thm}

Recall from Proposition \ref{prop:n01} that $\mathcal A(F_{g,n})$ is generated by $4^{g}-1$ knots when $n = 0, 1$.  We obtain a count for $n>1$ from Theorem \ref{thm:simplegenerators}.

%%%%%%%%%%%%%%%
%%%--- The Theorem ---%%%
%%%%%%%%%%%%%%%

%%%--- thm:finitegen ---%%%
\begin{cor} \label{thm:finitegen}
For $n > 1$ the arc algebra $\mathcal A(F_{g,n})$ is generated by a set of $n(4^{g}-1)$ simple knots and $\frac{n(n-1)}{2} (4^{g})$ simple arcs.
\end{cor}

\begin{proof}
Recall from Lemma~\ref{lem:determinehandle} that simple arcs with total complexity $(0,0,0)$ are determined by the handles they meet and their endpoints.   There are  $\binom{n}{2} = \frac{n(n-1)}{2}$ choices of distinct endpoints.  A simple arc with total complexity $(0,0,0)$ does not pass through any puncture handles.  For every pair of overlapping handles, such an arc can either go through neither handle, exactly one of the two handles, or both handles in a good way.     This gives four choices for each pair of handles and thus there are $ \binom{n}{2}4^{g}$ simple arcs in the generating set.

Simple knots with total complexity $(0,0,0)$ are also determined by the handles that they meet, and a similar counting argument as the one for simple arcs also applies.  
Taking away the unknot that passes through no handles whatsoever, it follows that there are $4^{g} -1$ simple knots with total complexity $(0,0,0)$.    A simple knot which lies in $F_{g, n-1}^*$ with total complexity $(0,0,1)$ passes through exactly one of the $n-1$ puncture handles.   For every pair of overlapping handles there are again four possibilities.  However, we must discount those simple knots that pass through one puncture handle but no handles from an overlapping pair, since they enclose either a disk or a once-punctured disk and hence are not simple.  Thus there are $(n-1) (4^{g} - 1)$ simple knots with total complexity $(0,0,1)$.  
In total, this gives $n(4^{g}-1)$ simple knots in the generating set.  
\end{proof}

\subsection{Remarks}
Firstly, note that the number of generators from Corollary~\ref{thm:finitegen} is sharp for the cases where the surface is a 2- or 3-punctured sphere or a 1-punctured torus, as was demonstrated in the companion paper \cite{BobbPeiferKennedyWong} by determining the relations between the generators.  

In the case of a sphere with $n$ punctures, the arc algebra is generated by a set of arcs with exactly one generator between every pair of punctures. 
While the sphere arc algebras have a finite generating set of only simple arcs, in general the authors do not know of a way to generate the arc algebra with only arcs when $A$ is generic and the surface has positive genus.   

However, if one is willing to divide by $A-A^{-1}$, then the arc algebra for any surface can be generated with only arcs.  The trick is to take the puncture-skein relation and its reverse, and then solve for any of the two diagrams which do not intersect at the puncture.    Thus the set of all zero complexity arcs (those with distinct endpoints and those with both endpoints at the same puncture) would generate the arc algebra.
As there are $n$ punctures and two ways for a pair of strands to meet at a puncture, the arc algebra $\mathcal A(F_{g,n})$ is generated by $\frac{n(n-1)}{2}(4^{g}) + 2n(4^{g}) = \frac{n(n+3)}{2} (4^g)$ arcs, so long as $A-A^{-1}$ is invertible. 

On the other hand, it is impossible to generate the arc algebra with only links when $n>1$, no matter what extra conditions we place on $A$.  Consider all the arcs that could be generated by a set of links.  These arcs could only be created using the puncture-skein relation, and each such arc would have an even number of strands connected to each puncture.  In particular, any connected arc whose endpoints are distinct cannot be generated by a set of links.

Finally, we remark that Bullock in \cite{Bu99} provides a generating set for the skein algebra consisting of $2^{(2g+n-1)} - 1$ elements.
Przytycki and Sikora show in \cite{PrzSikora} that this number, which is exponential in both $g$ and $n$,  is minimal for the skein algebra.  Corollary~\ref{thm:finitegen} says that the arc algebra is generated using $n  (4^{g}-1) + \frac{n(n-1)}{2} \, (4^{g})$, which is exponential in $g$ but only quadratic in $n$.   In particular, when $n \ge 5$, the number of generators for the arc algebra described here is strictly less than the number of generators needed for the usual skein algebra.

%%%%%%%%%%%%%%%%%%%%%%%%%%%%%%%%%%%%%%%%%%%%%%%%%%
%%%--- References ---%%%
%%%%%%%%%%%%%%%%%%%%%%%%%%%%%%%%%%%%%%%%%%%%%%%%%%
\bibliographystyle{plain}

\end{document}